\documentclass[11pt,reqno]{amsart}
\usepackage{amsmath}
\usepackage{amssymb}
\usepackage{mathrsfs}
\usepackage{amsthm}
\usepackage[draft=true]{hyperref}
\usepackage[height=190mm,width=130mm]{geometry} 
\theoremstyle{plain}
\newtheorem{mydef}{Definition}[section]
\newtheorem{mythm}{Theorem}[section]
\newtheorem{mylemma}{Lemma}[section]
\newtheorem{mycor}{Corollary}[section]
\newtheorem{Ex}{Example}[section]
\begin{document}
\title{Solving Sequential Linear \textit{M}-Fractional Differential Equations With Constants Coefficients}
\author{ V.Padmapriya$^1$ }
\address{ $^1$ Research Scholar, \\ VIT University,Chennai Campus,\\ India} 
\email{v.padmapriya2015@vit.ac.in}
\author{\ M.Kaliyappan $^2$}
\address{$^2$ Division of Mathematics, School of Advanced Sciences,\\
VIT University,Chennai Campus,\\India}
\email{kaliyappan.m@vit.ac.in}

\begin{abstract}
Fractional calculus is a powerful and effective tool for modelling nonlinear systems. The \textit{M}-derivative is the generalisation of alternative fractional derivative introduced by Katugampola\cite{r6}.  This \textit{M}-derivative obey the properties of integer calculus. In this paper, we present the method for solving \textit{M}-fractional sequential linear differential equations with constant coefficients for $\alpha\geq0$ and $\beta>0$. Existence and Uniqueness of the solutions for the n$^{th}$ order sequential linear \textit{M}-fractional differential equations are discussed in detail. We have present illustration for homogeneous and non homogeneous case.
\end{abstract}


\maketitle
\section{Introduction}
While L'Hospital has proposed the idea of fractional derivative in the $17^{th}$ century, several researchers concerted fractional derivative in the recent centuries. Riemann-Liouville, Caputo and other fractional derivatives are defined on the basis of fractional integral form \cite{r8,r12,r13}.\\
Recently, Khalil et al.\cite{r7} and Katugampola \cite{r6} proposed fractional derivatives in the limit form as in usual derivative such as conformable fractional derivative and alternative fractional derivative. Based on these derivative, Sousa and Oliveira \cite{r15} introduced \textit{M}-fractional derivatives which satisfies properties of integer-order calculus.\\
Theory and applications of the sequential linear fractional differential equations involving Hadamard, Riemann-Liouville, Caputo and Conformable derivatives have been investigated in \cite{r1,r2,r3,r4,r9,r10,r11}.\\
Lately, Gokdogan et al \cite{r5} have proved existence and uniqueness theorems for solving sequential linear conformable fractional differential equations. Unal et al \cite{r14} provide method to solve sequential linear conformable fractional differential equations with constants coefficients. In this work, We present Existence and Uniqueness theorems and solutions of sequential linear \textit{M}-fractional differential equations.\\
The arrangement of this paper is as following: In section 2 we present the concept of \textit{M}-fractional derivative. In section 3 we provide existence and uniqueness theorems for sequential Linear \textit{M}-fractional differential equations. In section 4 we propose the solutions of sequential Linear \textit{M}-fractional differential equations. In section 5 we present solutions of Non-homogeneous case. Finally, Conclusion is present in section 6.\\

\section{\textit{M}-Fractional Calculus}
In this section, we give some necessary definitions and theorems of M – derivative which are explained in \cite{r15}.

\begin{mydef}
Let $ f:[0,\infty)\rightarrow \mathbb{R}$ be a function and $t>0$. Then for $0<\alpha<1$, the \textit{M}-fractional derivative of $f$ of order $\alpha$ is defined as

	$$D^{\alpha,\beta}_{M} f(t)=\lim_{\epsilon \to 0} \frac{⁡f(t\mathbb{E}_{\beta}(\epsilon t^{-\alpha})-f(t)}{\epsilon}$$
	
Where $\mathbb{E}_{\beta} (.), \beta>0$ is Mittag-Leffler function with one parameter.
	
If $f$ is \textit{M}-differentiable in some interval $(0,a),a>0$ and $$\lim_{t \to 0^+}D^{\alpha,\beta}_{M} f(t)$$⁡ exists, then we have 
		$$D^{\alpha,\beta}_{M}f(0)=\lim_{t\to 0^+}D^{\alpha,\beta}_{M} f(t)$$
\end{mydef}

\begin{mythm}
Let $0<\alpha \leq 1,\beta>0, a,b\in\mathbb{R}$ and $f,g$ be \textit{M} - differentiable at a point $t>0$. Then 
\begin{enumerate}
\item $D^{\alpha,\beta}_{M} (af+bg)(t)=aD^{\alpha,\beta}_{M} f(t)+bD^{\alpha,\beta}_{M} g(t)$ for all $a,b\in \mathbb{R}$.
\item $D^{\alpha,\beta}_{M}(f\cdot g) (t)= f(t) D^{\alpha,\beta}_{M}g(t) +g(t) D^{\alpha,\beta}_{M}f(t)$
\item  $ D^{\alpha,\beta}_{M}(\frac{f}{g})(t)=\frac{g(t) D^{\alpha,\beta}_{M} f(t)-f(t) D^{\alpha,\beta}_{M} g(t))}{[g(t)]^{2}}$
\item $D^{\alpha,\beta}_{M}(c)=0$, where $f(t)=c$ is a constant. 
\item $D^{\alpha,\beta}_{M}t^{a}=\frac{a}{\Gamma(\beta+1)} t^{a-\alpha}$
\item Moreover, $f$ is differentiable, then $D^{\alpha,\beta}_{M}f(t)=\frac{t^{1-\alpha}}{\Gamma(\beta+1)} \frac{d f(t)}{dt}$  
\end{enumerate}
Additionally,\textit{M}-derivatives of certain functions as follows:
\begin{enumerate}
\item  $D^{\alpha,\beta}_{M}(sin(\frac{1}{\alpha}t^{\alpha}))=\frac{cos(\frac{1}{\alpha}t^{\alpha})}{\Gamma(\beta+1)}$
\item $D^{\alpha,\beta}_{M}(cos(\frac{1}{\alpha}t^{\alpha}))=\frac{-sin(\frac{1}{\alpha}t^{\alpha})}{\Gamma(\beta+1)}$
\item $D^{\alpha,\beta}_{M}(e^{\frac{t^{\alpha}}{\alpha}})=\frac{e^{\frac{t^{\alpha}}{\alpha}}}{\Gamma(\beta+1)}$ 
\end{enumerate}
\end{mythm}

\begin{mythm}
Let $a>0$ and $t\geq a$. Then, the \textit{M}-integral of order $α$ of a function $f$ is defined by
$$_{M} I^{\alpha,\beta}_{a}f(t)= \Gamma(\beta+1) \int_{a}^{t} \frac{f(x)}{x^{1-\alpha}}dx$$
\end{mythm}

\begin{mythm}
Let $a\geq 0$ and $0<\alpha<1$. Also let $f$ be a continuous function such that there exists $_{M} I^{\alpha,\beta}_{a}f$. Then
$$D^{\alpha,\beta}_{M} (_{M} I^{\alpha,\beta}_{a} f(t))=f(t)$$
\end{mythm}

\begin{mythm}
Let $f,g:[a,b]\rightarrow \mathbb{R}$ be two functions such that $f,g$ are differentiable and $0<\alpha<1$. Then
$$\int_{a}^{b}f(x)D^{\alpha,\beta}_{M}g(x)d_{\alpha}x=f(x)g(x)\Bigr|_{a}^{b}-\int_{a}^{b}g(x)D^{\alpha,\beta}_{M}f(x)d_{\alpha}x$$
where $d_{\alpha}x=\frac{\Gamma(\beta+1)}{x^{1-\alpha}}dx$
\end{mythm}

\section{Existence and Uniqueness Theorem}
Let linear sequential \textit{M}-fractional differential equation of order $n\alpha$
\begin{equation}
^{n}D^{\alpha,\beta}_{M}y+p_{n-1}(t)^{n-1}D^{\alpha,\beta}_{M}y+...+p_{2}(t)^{2}D^{\alpha,\beta}_{M}y+p_{1}(t)D^{\alpha,\beta}_{M}y+p_{0}(t)y=0
\end{equation}
where $^{n}D^{\alpha,\beta}_{M}y=D^{\alpha,\beta}_{M}D^{\alpha,\beta}_{M}...D^{\alpha,\beta}_{M}y,$ ($n$ times)\\
Similarly, non-homogeneous fractional differential equation with \textit{M}-derivative is
\begin{equation}
^{n}D^{\alpha,\beta}_{M}y+p_{n-1}(t)^{n-1}D^{\alpha,\beta}_{M}y+...+p_{2}(t)^{2}D^{\alpha,\beta}_{M}y+p_{1}(t)D^{\alpha,\beta}_{M}y+p_{0}(t)y=f(t)
\end{equation}
We define an $n^{th}$-order differential operator for eqn. (1) as following
\begin{equation}
L_{\alpha,\beta}[y]=^{n}D^{\alpha,\beta}_{M}y+p_{n-1}(t)^{n-1}D^{\alpha,\beta}_{M}y+...+p_{2}(t)^{2}D^{\alpha,\beta}_{M}y+p_{1}(t)D^{\alpha,\beta}_{M}y+p_{0}(t)y=0
\end{equation}

\begin{mythm}
Let $\Gamma(\beta+1)t^{\alpha-1} p(t),\Gamma(\beta+1)t^{\alpha-1}f(t)\in C(a,b)$ and let $y$ be \textit{M}-differentiable for $0<\alpha\leq 1$ and $\beta>0$. Then the initial value problem
\begin{equation}
D^{\alpha,\beta}_{M}y+p(t)y=f(t)
\end{equation}
\begin{equation}
y(t_{0})=y_{0}
\end{equation}
has exactly one solution on the interval $(a,b)$ where $t_{0}\in(a,b)$    
\end{mythm}
\begin{proof}
Using property (6) in Theorem 2.1, we have
$$D^{\alpha,\beta}_{M}y+p(t)y=f(t)$$
$$\frac{t^{1-\alpha}}{\Gamma(\beta+1)} y'+p(t)y=f(t)$$
$$y'+\Gamma(\beta+1)t^{\alpha-1}p(t)y=\Gamma(\beta+1)t^{\alpha-1}f(t)$$
The proof is clear from classical linear fundamental theorem existence and uniqueness.
\end{proof}

\begin{mythm}
If  $\Gamma(\beta+1)t^{\alpha-1}p_{n-1}(t),...,\Gamma(\beta+1)t^{\alpha-1}p_{1}(t),\Gamma(\beta+1)t^{\alpha-1}p_{0}(t),\Gamma(\beta+1)t^{\alpha-1}f(t)\in C(a,b)$ and $y$ be $n$ times \textit{M}-differentiable function, then a solution $y(t)$ of the initial value problem
\begin{equation}
^{n}D^{\alpha,\beta}_{M}y+p_{n-1}(t)^{n-1}D^{\alpha,\beta}_{M}y+...+p_{2}(t)^{2}D^{\alpha,\beta}_{M}y+p_{1}(t)D^{\alpha,\beta}_{M}y+p_{0}(t)y=f(t)
\end{equation}
\begin{equation}
y(t_{0})=y_{0},D^{\alpha,\beta}_{M}y(t_{0})=y_{1},...,^{n-1}D^{\alpha,\beta}_{M}y(t_{0})=y_{n-1}, a<t_{0}<b
\end{equation}
\end{mythm}
\begin{proof}
The existence of a local solution is obtained by transform our problem into the first order system of differential equations. So, we introduce new variables
$$x_{1}=y,x_{2}=D^{\alpha,\beta}_{M}y,x_{3}=^{2}D^{\alpha,\beta}_{M}y,,...,x_{n}=^{n-1}D^{\alpha,\beta}_{M}y$$
In this, we have
$$D^{\alpha,\beta}_{M}x_{1}=x_{2}$$
$$D^{\alpha,\beta}_{M}x_{2}=x_{3}$$
$$\vdots$$
$$D^{\alpha,\beta}_{M}x_{n-1}=x_{n}$$
$$D^{\alpha,\beta}_{M}x_{n}=-p_{n-1}x_{n}-...-p_{2} x_{3}-p_{1} x_{2}-p_{0} x_{1}+f(t)$$
The above equations can be written as the following
$$D^{\alpha,\beta}_{M}\underbrace{\begin{bmatrix}
x_{1}\\x_{2}\\\vdots \\x_{n-1}\\x_{n}
\end{bmatrix}}_{X(t)}+\underbrace{\begin{bmatrix}
0&-1&0&0&\cdots&0\\0&0&-1&0&\cdots&0\\\vdots&\vdots&\vdots&\vdots&\cdots&\vdots\\0&0&0&0&\cdots&-1\\p_{0}&p_{1}&p_{2}&p_{3}&\cdots&p_{n-1}\end{bmatrix}}_{P(t)}
\begin{bmatrix} x_{1}\\x_{2}\\\vdots \\x_{n-1}\\x_{n} \end{bmatrix}=\underbrace{\begin{bmatrix} 0\\0\\\vdots \\0\\f(t) \end{bmatrix}}_{F(t)}$$
$$D^{\alpha,\beta}_{M}X(t)+P(t)X(t)=F(t)$$
$$X'(t)+\Gamma(\beta+1)t^{\alpha-1}P(t)X(t)=\Gamma(\beta+1)t^{\alpha-1}F(t)$$
The existence and uniqueness of solution (6)-(7) follows from classical theorems on existence and uniqueness for system equation.
\end{proof}
\begin{mythm}
If $y_{1}$ and $y_{2}$ are $n$ times \textit{M}-differentiable functions and $c_{1},c_{2}$ are arbitrary numbers, then $L_{\alpha,\beta}$ is linear.
$$i.e,\	 L_{\alpha,\beta}[c_{1} y_{1}+c_{2} y_{2}]=c_{1} L_{\alpha,\beta}[y_{1} ]+c_{2}L_{\alpha,\beta}[y_{2}]$$
\end{mythm}
\begin{proof}
We can easily derived the proof of this theorem by applying same procedure  in Theorem-4.3 \cite{r5} to \textit{M}- derivative.
\end{proof}

\begin{mythm}
If $y_{1},y_{2}...y_{n}$ are the solutions of equation $L_{\alpha,\beta}[y]=0$ and $c_{1},c_{2}...c_{n}$ are arbitrary constants, then the linear combination $y(t)=c_{1} y_{1}+c_{2} y_{2}+...+c_{n}y_{n}$ is also solution of $L_{\alpha,\beta}[y]=0$.
\end{mythm}
\begin{proof}
We can easily derived the proof of this theorem by applying same procedure in Theorem-4.4\cite{r5} to \textit{M}- derivative.
\end{proof}

\begin{mydef}
For $n$ functions $y_{1},y_{2}...y_{n}$, we define the \textit{M}-Wronskain of these function to be the determinant
$$W_{\alpha,\beta}(t)=\begin{vmatrix}
y_1&y_{2}&\cdots&y_{n}\\
D^{\alpha,\beta}_{M}y_{1}&D^{\alpha,\beta}_{M}y_{2}&\cdots&D^{\alpha,\beta}_{M}y_{n}\\
\vdots&\vdots&\cdots&\vdots\\
^{n-1}D^{\alpha,\beta}_{M}y_{1}&^{n-1}D^{\alpha,\beta}_{M}y_{2}&\cdots&^{n-1}D^{\alpha,\beta}_{M}y_{n}
\end{vmatrix}$$
\end{mydef}

\begin{mythm}
Let $y_{1},y_{2}...y_{n}$ be $n$ solutions of $L_{\alpha,\beta}[y]=0$. If there is a $t_{0}\in(a,b)$ such that $W_{\alpha,\beta}(t_{0})\neq0$, then ${y_{1},y_{2}...y_{n}}$ is a fundamental set of solutions.
\end{mythm}
\begin{proof}
We need to show that if $y(t)$ is a solution of $L_{\alpha,\beta}[y]=0$, then we can write $y(t)$ as a linear combination of $y_{1},y_{2}...y_{n}$. 
$$i.e, y=c_{1}y_{1}+c_{2}y_{2}+...+c_{n}y_{n}$$
so the problem reduces to finding the constants $c_{1},c_{2},...c_{n}$.These constants are found by solving the following linear system of $n$ equations
$$c_{1} y_{1}(t_{0})+c_{2}y_{2}(t_{0})+...+c_{n}y_{n}(t_{0})=y(t_{0})$$
$$c_{1}D^{\alpha,\beta}_{M} y_{1}(t_{0})+c_{2}D^{\alpha,\beta}_{M}y_{2}(t_{0})+...+c_{n}D^{\alpha,\beta}_{M}y_{n}(t_{0})=D^{\alpha,\beta}_{M}y(t_{0})$$
$$\vdots\\$$
$$c_{1}^{n-1}D^{\alpha,\beta}_{M} y_{1}(t_{0})+c_{2}^{n-1}D^{\alpha,\beta}_{M}y_{2}(t_{0})+...+c_{n}^{n-1}D^{\alpha,\beta}_{M}y_{n}(t_{0})=^{n-1}D^{\alpha,\beta}_{M}y(t_{0})$$
Using Cramer’s rule, we can find
$$c_{i}=\frac{W_{\alpha,\beta}^{i}(t_{0})}{W_{\alpha,\beta}(t_{0})}, 1\leq i \leq n $$ 
Since $W_{\alpha,\beta}(t_{0})\neq 0$, it follows that $c_{1},c_{2},...c_{n}$ exist.  
\end{proof}

\begin{mythm}
Let $y_{1},y_{2}...y_{n}$ be $n$ solutions of $L_{\alpha,\beta}[y]=0$. Then
\begin{enumerate}
\item $W_{\alpha,\beta}(t)$ satisfies the differential equation$D^{\alpha,\beta}_{M}W_{\alpha,\beta}+p_(n-1) W_{\alpha,\beta}=0$
\item If $t_{0}$ is any point in $(a,b)$, then 
$$W_{\alpha,\beta}(t)=W_{\alpha,\beta}(t_{0})e^{-\Gamma(\beta+1)\int_{t_{0}}^{t}x^{\alpha-1}p_{n-1}(x)dx}$$
Further, if $W_{\alpha,\beta}(t_{0})\neq0$ then $W_{\alpha,\beta}(t)\neq0$ for all $t\in(a,b)$
\end{enumerate} 
\end{mythm}
\begin{proof}
(1)Let us introduce new variables
$$x_{1}=y,x_{2}=D^{\alpha,\beta}_{M}y,x_{3}=^{2}D^{\alpha,\beta}_{M}y,...,x_{n}=^{n-1}D^{\alpha,\beta}_{M}y$$
From this, we have
$$D^{\alpha,\beta}_{M}x_{1}=x_{2}$$
$$D^{\alpha,\beta}_{M}x_{2}=x_{3}$$
$$\vdots$$
$$D^{\alpha,\beta}_{M}x_{n-1}=x_{n}$$
$$D^{\alpha,\beta}_{M}x_{n}=-p_{n-1}x_{n}-...-p_{2} x_{3}-p_{1} x_{2}-p_{0} x_{1}$$
$$D^{\alpha,\beta}_{M}\underbrace{\begin{bmatrix}
x_{1}\\x_{2}\\\vdots \\x_{n-1}\\x_{n}
\end{bmatrix}}_{X(t)}=\underbrace{\begin{bmatrix}
0&-1&0&0&\cdots&0\\0&0&-1&0&\cdots&0\\\vdots&\vdots&\vdots&\vdots&\cdots&\vdots\\0&0&0&0&\cdots&-1\\-p_{0}&-p_{1}&-p_{2}&-p_{3}&\cdots&-p_{n-1}\end{bmatrix}}_{P(t)}
\begin{bmatrix} x_{1}\\x_{2}\\\vdots \\x_{n-1}\\x_{n} \end{bmatrix}=\underbrace{\begin{bmatrix} 0\\0\\\vdots \\0\\f(t) \end{bmatrix}}_{F(t)}$$
$$D^{\alpha,\beta}_{M}X(t)=P(t)X(t)$$
We have
$$D^{\alpha,\beta}_{M}W_{\alpha,\beta}(t)=(a_{11}+a_{22}+...+a_{nn})W_{\alpha,\beta}(t)$$
In our case
$$a_{11}+a_{22}+...+a_{nn}=-p_{n-1}(t)$$
So, $$D^{\alpha,\beta}_{M}W_{\alpha,\beta}(t)+p_{n-1}(t)W_{\alpha,\beta}(t)=0$$
(2)The above differential equation can be solved by the method of integrating factor, we have 
	$$W_{\alpha,\beta}(t)=W_{\alpha,\beta}(t_{0})e^{-\Gamma(\beta+1)\int_{t_{0}}^{t}x^{\alpha-1}p_{n-1}(x)dx}$$
Thus the proof of theorem is completed.
\end{proof}

\begin{mythm}
If  $\{y_{1},y_{2}...y_{n}\}$ is a fundamental set of solutions of $L_{\alpha,\beta}[y]=0$ where $\Gamma(\beta+1)t^{\alpha-1}p_{n-1}(t)...\Gamma(\beta+1)t^{\alpha-1}p_{1}(t),\Gamma(\beta+1)t^{\alpha-1}p_{0}(t)\in C(a,b)$,then $W_{\alpha,\beta}(t)\neq 0$ for all $t\in (a,b)$.
\end{mythm}
\begin{proof}
By applying procedure in Theorem-4.8 [5] to \textit{M}-derivative, we can easily prove this theorem.
\end{proof}
\begin{mythm}
Let $\Gamma(\beta+1)t^{\alpha-1}p_{n-1}(t),...,\Gamma(\beta+1)t^{\alpha-1}p_{1}(t), \Gamma(\beta+1)t^{\alpha-1}p_{0}(t)\in C(a,b)$. The solution set $\{y_{1},y_{2},...,y_{n}\}$ is a fundamental set of solutions to the equation $L_{\alpha,\beta}[y]=0$ if and only if the functions $y_{1},y_{2},...,y_{n}$ are linearly independent.
\end{mythm}
\begin{proof}
By applying procedure in Theorem-4.9 \cite{r5} to \textit{M}-derivative, we can easily prove this theorem.
\end{proof}

\begin{mythm}
Let $y_{1},y_{2},...,y_{n}$ be a fundamental set of solutions of the equation (1) and $y_{p}$ be any particular solution of the non homogeneous equation (2). Then the general solution of the equation is $y=c_{1}y_{1}+c_{2}y_{2}+...+c_{n}y_{n}+y_{p}$
\end{mythm}
\begin{proof}
Let $L_{\alpha,\beta}$ be the differential operator and $y(t)$ and $y_p(t)$ be the solutions of the non homogeneous equation $L_{\alpha,\beta}[y]=f(t)$. If we take $u(t)=y(t)-y_{p}(t)$, then by linearity of $L_{\alpha,\beta}$ we have,
$$L_{\alpha,\beta}[u]=L_{\alpha,\beta}[y(t)-y_{p}(t)]=L_{\alpha,\beta}[y(t)]-L_{\alpha,\beta}[y_p (t)]=f(t)-f(t)=0$$
Then $u(t)$ is a solution of the homogenous equation $L_{\alpha,\beta}[y]=0$. Then by Theorem 3.4
$$u(t)=c_{1}y_{1}(t)+c_{2}y_{2}(t)+...+c_{n}y_{n}(t)$$
i.e, 			$$y(t)-y_{p}(t)=c_{1}y_{1}(t)+c_{2}y_{2}(t)+...+c_{n}y_{n}(t)$$
Then			$$ y(t)=c_{1} y_{1}(t)+c_{2} y_{2}(t)+...+c_{n}y_{n}(t)+y_{p}(t)$$
\end{proof}

\section{Solution of Homogeneous Case}
Consider the $n$ times \textit{M}-differentiable function $y$ for $\alpha \in (0,1]$ and $\beta>0$. The homogeneous sequential linear fractional differential equation with \textit{M}-derivative is 
\begin{equation}
^{n}D^{\alpha,\beta}_{M}y+p_{n-1}(t)^{n-1}D^{\alpha,\beta}_{M}y+...+p_{2}(t)^{2}D^{\alpha,\beta}_{M}y+p_{1}(t)D^{\alpha,\beta}_{M}y+p_{0}(t)y=0
\end{equation}
where $^{n}D^{\alpha,\beta}_{M}y=D^{\alpha,\beta}_{M}D^{\alpha,\beta}_{M}...D^{\alpha,\beta}_{M}y$ $n$ times, and the coefficients $p_{0},p_{1},...,p_{n-1}$ are real constants.\\
We define an $n^{th}$-order differential operator for eqn. (1) as following
\begin{equation}
L_{\alpha,\beta}[y]=^{n}D^{\alpha,\beta}_{M}y+p_{n-1}^{n-1}D^{\alpha,\beta}_{M}y+...+p_{2}^{2}D^{\alpha,\beta}_{M}y+p_{1}D^{\alpha,\beta}_{M}y+p_{0}y=0
\end{equation}
If $y_{1}(t),y_{2}(t),...,y_{n}(t)$ are linearly independent solutions of Eqn.(1), then general solution is 
	$$y=c_{1} y_{1}(t)+c_{2} y_{2}(t)+...+c_{n}y_{n}(t)$$
where $c_{1},c_{2}...c_{n}$ are arbitrary constants.

\begin{mylemma}
Suppose that  $L_{\alpha,\beta}[.]$ is a linear operator with constant coefficients and $\alpha\in(0,1]$ and $\beta>0$, then for $t>0$
	 $$L_{\alpha,\beta}[e^{\frac{r\Gamma(\beta+1)}{\alpha}t^{\alpha}}]=P_{n}(r)[e^{\frac{r\Gamma(\beta+1)}{\alpha}t^{\alpha}}]$$
Where $P_{n}(r)=r^{n}+P_{n-1}r^{n-1}+...+P_{0}$ and $r$ is a real or complex constant
\end{mylemma}
\begin{proof}
\textit{M}-derivatives of $y=e^{\frac{r\Gamma(\beta+1)}{\alpha}t^{\alpha}}$  are 
\begin{equation}
D^{\alpha,\beta}_{M}y=re^{\frac{r\Gamma(\beta+1)}{\alpha}t^{\alpha}},  ^{2}D^{\alpha,\beta}_{M} y=r^2 e^{\frac{r\Gamma(\beta+1)}{\alpha}t^{\alpha}},...,^{n}D^{\alpha,\beta}_{M} y=r^n e^{\frac{r\Gamma(\beta+1)}{\alpha}t^{\alpha}}
\end{equation}
We substitute $y=e^{\frac{r\Gamma(\beta+1)}{\alpha}t^{\alpha}}$ and Eqn.(10) in $L_{\alpha,\beta}[y]$
$$L_{\alpha,\beta}[e^{\frac{r\Gamma(\beta+1)}{\alpha}t^{\alpha}}]=(^{n}D^{\alpha,\beta}_{M}+p_{n-1}^{n-1}D^{\alpha,\beta}_{M}+...+p_{2}^{2}D^{\alpha,\beta}_{M}+p_{1}D^{\alpha,\beta}_{M}+p_{0})e^{\frac{r\Gamma(\beta+1)}{\alpha}t^{\alpha}}$$
$$=(r^{n}+P_{n-1}r^{n-1}+...+P_{0})e^{\frac{r\Gamma(\beta+1)}{\alpha}t^{\alpha}}$$
$$L_{\alpha,\beta}[e^{\frac{r\Gamma(\beta+1)}{\alpha}t^{\alpha}}]=p_{n}(r)e^{\frac{r\Gamma(\beta+1)}{\alpha}t^{\alpha}}$$
Hence, the proof is completed.
\end{proof}

The solution to the equation (8) is  $y=e^{\frac{r\Gamma(\beta+1)}{\alpha}t^{\alpha}}$.
\\It follows from Eqn.(9) and Lemma 3.1 that 
$$L_{\alpha,\beta}[e^{\frac{r\Gamma(\beta+1)}{\alpha}t^{\alpha}}]=p_{n}(r)e^{\frac{r\Gamma(\beta+1)}{\alpha}t^{\alpha}}=0$$
Where $P_{n}(r)=r^{n}+P_{n-1}r^{n-1}+...+P_{0}$ is called as the characteristic polynomial. For all $r$, we have $e^{\frac{r\Gamma(\beta+1)}{\alpha}t^{\alpha}}\neq0$. Hence $P_{n}(r)=0$.
\\Here
\begin{equation}
			 r^{n}+P_{n-1}r^{n-1}+...+P_{0}=0		
\end{equation}
is called as the characteristic equation.

\begin{mylemma}
Let $r$ be a root of the characteristic equation (11), then
$$\frac{\partial}{\partial r}\{L_{\alpha,\beta}[e^{\frac{r\Gamma(\beta+1)}{\alpha}t^{\alpha}}]\}=L_{\alpha,\beta} [\frac{\partial}{\partial r}e^{\frac{r\Gamma(\beta+1)}{\alpha}t^{\alpha}}]$$
and
		$\frac{\partial^{l}}{\partial r^{l}}e^{\frac{r\Gamma(\beta+1)}{\alpha}t^{\alpha}} = (\frac{\Gamma(\beta+1)}{\alpha}t^{\alpha})^{l}e^{\frac{r\Gamma(\beta+1)}{\alpha}t^{\alpha}}$ where $l$ is integer.
\end{mylemma}
\begin{proof}
From Theorem 3.3 it follows that $L_{\alpha,\beta}[.]$ is linear and also  $\frac{\partial}{\partial r}$ is linear by property of classical derivative. Hence
$$\frac{\partial}{\partial r}\{L_{\alpha,\beta}[e^{\frac{r\Gamma(\beta+1)}{\alpha}t^{\alpha}}]\}=L_{\alpha,\beta} [\frac{\partial}{\partial r}e^{\frac{r\Gamma(\beta+1)}{\alpha}t^{\alpha}}]$$
Additionally, from classical derivative, it follows that
$$\frac{\partial^{l}}{\partial r^{l}}e^{\frac{r\Gamma(\beta+1)}{\alpha}t^{\alpha}} = (\frac{\Gamma(\beta+1)}{\alpha}t^{\alpha})^{l}e^{\frac{r\Gamma(\beta+1)}{\alpha}t^{\alpha}}$$ 
\end{proof}

\begin{mylemma}
If $r_{1}$ is a root of multiplicity of $\mu_{1}$of the characteristic equation (11), then the functions $y_{1,l}(t)$ , where $l=0,1,...,\mu_{l-1}$ such that
$$y_{1,l} = (\frac{\Gamma(\beta+1)}{\alpha}t^{\alpha})^{l}e^{\frac{r_{1}\Gamma(\beta+1)}{\alpha}t^{\alpha}}$$ are solutions of Eq.(8).
\end{mylemma}
\begin{proof}
Consider $L_{\alpha,\beta}[e^{\frac{r\Gamma(\beta+1)}{\alpha}t^{\alpha}}]=p_{n}(r)e^{\frac{r\Gamma(\beta+1)}{\alpha}t^{\alpha}}$. From Lemma 4.2 and applying classical Leibniz rule it follows that 
$$\Bigg\{L_{\alpha,\beta}\bigg[\frac{\partial^l}{\partial r^l}e^{\frac{r\Gamma(\beta+1)}{\alpha}t^{\alpha}}\bigg]\Bigg\}_{r=r_{1}}=\Bigg\{\frac{\partial^l}{\partial r^l}\Bigg[L_{\alpha,\beta}\bigg[e^{\frac{r\Gamma(\beta+1)}{\alpha}t^{\alpha}}\bigg]\Bigg]\Bigg\}_{r=r_{1}}=\Bigg\{\frac{\partial^l}{\partial r^l}\bigg[p_{n}(r)e^{\frac{r\Gamma(\beta+1)}{\alpha}t^{\alpha}}\bigg]\Bigg\}_{r=r_{1}}$$
$$\Bigg\{L_{\alpha,\beta}\bigg[\frac{\partial^l}{\partial r^l}e^{\frac{r\Gamma(\beta+1)}{\alpha}t^{\alpha}}\bigg]\Bigg\}_{r=r_{1}}=\sum_{j=0}^{l}
\begin{pmatrix} l\\j\end{pmatrix}\bigg[\frac{\partial^{l-j}}{\partial r^{l-j}}e^{\frac{r\Gamma(\beta+1)}{\alpha}t^{\alpha}}\bigg]_{r=r_{1}}\frac{\partial^{j}}{\partial r^{j}}\big[P_{n}(r)\big]_{r=r_{1}}$$
Since $\frac{\partial^{j}}{\partial r^{j}}\big[P_{n}(r)\big]_{r=r_{1}}=0$ for $j=0,1,...,\mu_{1}-1$
$$\Bigg\{L_{\alpha,\beta}\bigg[\frac{\partial^l}{\partial r^l}e^{\frac{r\Gamma(\beta+1)}{\alpha}t^{\alpha}}\bigg]\Bigg\}_{r=r_{1}}=0$$
From Lemma 4.2
$$\Bigg\{L_{\alpha,\beta}\bigg[\Big(\frac{\Gamma(\beta+1)}{\alpha}t^{\alpha}\Big)^{l}e^{\frac{r\Gamma(\beta+1)}{\alpha}t^{\alpha}}\bigg]\Bigg\}_{r=r_{1}}=0$$
$$L_{\alpha,\beta}\Big[y_{1,l}(t)\Big]=0$$
Hence $y_{1,l}(t)$ are solutions of Eq.(8).
\end{proof}

\begin{mycor}
Let $r_{j} ,j=1,2,...,k$ are distinct roots of multiplicity $\mu_{j},j=1,2,...,k$ of the characteristic Eq.(5). Then the following functions 
$$\bigcup_{j=1}^{k}\Bigg\{\Big(\frac{\Gamma(\beta+1)}{\alpha}t^{\alpha}\Big)^{l}e^{\frac{r_{j}\Gamma(\beta+1)}{\alpha}t^{\alpha}}\Bigg\}_{l=0}^{\mu_{j}-1}$$
\end{mycor}
\begin{proof}
Corollary 4.1 follows from Lemma 4.3 and Theorem 3.5.
\end{proof}

\begin{mylemma}
If  $r_{1}$  and $\bar{r_{1}}$  $(r_{1}=a+ib,b\neq0)$ are complex roots of multiplicity $\sigma_{1}$ of the characteristic equation (11), then for $l=0,1,...,\sigma_{1}-1$, the functions
$$y_{1,l}(t)=\Big(\frac{\Gamma(\beta+1)}{\alpha}t^{\alpha}\Big)^{l}e^{\frac{a\Gamma(\beta+1)}{\alpha}t^{\alpha}}\Big[cos\Big(\frac{b\Gamma(\beta+1)}{\alpha}t^{\alpha}\Big)+isin\Big(\frac{b\Gamma(\beta+1)}{\alpha}t^{\alpha}\Big)\Big]$$and
$$y_{2,l}(t)=\Big(\frac{\Gamma(\beta+1)}{\alpha}t^{\alpha}\Big)^{l}e^{\frac{a\Gamma(\beta+1)}{\alpha}t^{\alpha}}\Big[cos\Big(\frac{b\Gamma(\beta+1)}{\alpha}t^{\alpha}\Big)-isin\Big(\frac{b\Gamma(\beta+1)}{\alpha}t^{\alpha}\Big)\Big]$$
are linearly independent solutions of Eq.(8).
\end{mylemma}
\begin{proof}
Since $r_{1}=a+ib$ is a root of multiplicity $\sigma_{1}$ of the characteristic equation (11), From Lemma 4.3 and using Euler’s identity it follows that, the functions
$$y_{1,l}(t)=\Big(\frac{\Gamma(\beta+1)}{\alpha}t^{\alpha}\Big)^{l}e^{\frac{(a+ib)\Gamma(\beta+1)}{\alpha}t^{\alpha}}$$
i.e$$y_{1,l}(t)=\Big(\frac{\Gamma(\beta+1)}{\alpha}t^{\alpha}\Big)^{l}e^{\frac{a\Gamma(\beta+1)}{\alpha}t^{\alpha}}\Big[cos\Big(\frac{b\Gamma(\beta+1)}{\alpha}t^{\alpha}\Big)+isin\Big(\frac{b\Gamma(\beta+1)}{\alpha}t^{\alpha}\Big)\Big]$$
are solutions of the Eq.(8).
Similarly, for  $\bar{r_{1}}=a-ib$, the functions
$$y_{2,l}(t)=\Big(\frac{\Gamma(\beta+1)}{\alpha}t^{\alpha}\Big)^{l}e^{\frac{(a-ib)\Gamma(\beta+1)}{\alpha}t^{\alpha}}$$
i.e$$y_{2,l}(t)=\Big(\frac{\Gamma(\beta+1)}{\alpha}t^{\alpha}\Big)^{l}e^{\frac{a\Gamma(\beta+1)}{\alpha}t^{\alpha}}\Big[cos\Big(\frac{b\Gamma(\beta+1)}{\alpha}t^{\alpha}\Big)-isin\Big(\frac{b\Gamma(\beta+1)}{\alpha}t^{\alpha}\Big)\Big]$$
are solutions of the Eq.(8). Hence proof is completed.
\end{proof}

\begin{mycor}
If $\big\{r_{j},\bar{r_{j}}\big\}_{j=1}^{m}, r_{j}=a_{j}+ib_{j}, b_{j}\neq0 $ distinct $2m$ roots of multiplicity $\big\{\sigma_{j}\big\}_{j=1}^{m}$ of the characteristic equation (11),then, the functions
$$\bigcup_{j=1}^{m}\Bigg\{\Big(\frac{\Gamma(\beta+1)}{\alpha}t^{\alpha}\Big)^{l}e^{\frac{a_{j}\Gamma(\beta+1)}{\alpha}t^{\alpha}}\Big[cos\Big(\frac{b_{j}\Gamma(\beta+1)}{\alpha}t^{\alpha}\Big)+isin\Big(\frac{b_{j}\Gamma(\beta+1)}{\alpha}t^{\alpha}\Big)\Big]\Bigg\}_{l=0}^{\sigma_{j}-1}$$ and
$$\bigcup_{j=1}^{m}\Bigg\{\Big(\frac{\Gamma(\beta+1)}{\alpha}t^{\alpha}\Big)^{l}e^{\frac{a_{j}\Gamma(\beta+1)}{\alpha}t^{\alpha}}\Big[cos\Big(\frac{b_{j}\Gamma(\beta+1)}{\alpha}t^{\alpha}\Big)-isin\Big(\frac{b_{j}\Gamma(\beta+1)}{\alpha}t^{\alpha}\Big)\Big]\Bigg\}_{l=0}^{\sigma_{j}-1}$$
\end{mycor}
\begin{proof}
To prove corollary 4.2, it is sufficient to apply the Lemma 4.4 and Theorem 3.5.
\end{proof}

\begin{mythm}
If $\big\{r_{j}\big\}_{j=1}^{k}$ are distinct $k$ roots of multiplicity $\big\{\mu_{j}\big\}_{j=1}^{k}$ and $\big\{\lambda_{j},\bar{\lambda_{j}}\big\}_{j=1}^{m}, \lambda_{j}=a_{j}+ib_{j},b_{j}\neq 0$ are distinct $2m$ roots of multiplicity $\big\{\sigma_{j}\big\}_{j=1}^{m}$ of the characteristic equation (11) such that $\sum_{j=1}^{k}\mu_{j}+2\sum_{j=1}^{m}\sigma_{j}=n$, then the functions
$$\bigcup_{j=1}^{k}\Bigg\{\Big(\frac{\Gamma(\beta+1)}{\alpha}t^{\alpha}\Big)^{l}e^{\frac{r_{j}\Gamma(\beta+1)}{\alpha}t^{\alpha}}\Bigg\}_{l=0}^{\mu_{j}-1}$$
$$\bigcup_{j=1}^{m}\Bigg\{\Big(\frac{\Gamma(\beta+1)}{\alpha}t^{\alpha}\Big)^{l}e^{\frac{a_{j}\Gamma(\beta+1)}{\alpha}t^{\alpha}}\Big[cos\Big(\frac{b_{j}\Gamma(\beta+1)}{\alpha}t^{\alpha}\Big)+isin\Big(\frac{b_{j}\Gamma(\beta+1)}{\alpha}t^{\alpha}\Big)\Big]\Bigg\}_{l=0}^{\sigma_{j}-1}$$ and
$$\bigcup_{j=1}^{m}\Bigg\{\Big(\frac{\Gamma(\beta+1)}{\alpha}t^{\alpha}\Big)^{l}e^{\frac{a_{j}\Gamma(\beta+1)}{\alpha}t^{\alpha}}\Big[cos\Big(\frac{b_{j}\Gamma(\beta+1)}{\alpha}t^{\alpha}\Big)-isin\Big(\frac{b_{j}\Gamma(\beta+1)}{\alpha}t^{\alpha}\Big)\Big]\Bigg\}_{l=0}^{\sigma_{j}-1}$$
are the fundamental set of solutions of the equation (8).
\end{mythm}
\begin{proof}
The proof of the Theorem 4.1 is follows from Corollary 4.1, Corollary 4.2 and Theorem 3.5.
\end{proof}

\begin{Ex}
\begin{equation}
^{2}D^{\alpha,\beta}_{M}y+4D^{\alpha,\beta}_{M}y+3y=0
\end{equation}
The characteristic equation of (12) is $$r^{2}+4r+3=0$$
Therefore, the roots are $r=-3$ and $r=-1$
\\Hence, the general solution is 
$$y(t)=c_{1}e^{\frac{-3\Gamma(\beta+1)}{\alpha}t^{\alpha}}+c_{2} e^{\frac{-\Gamma(\beta+1)}{\alpha}t^{\alpha}}$$
\end{Ex}

\begin{Ex}
\begin{equation}
^{2}D^{\alpha,\beta}_{M}y-4D^{\alpha,\beta}_{M}y+4y=0
\end{equation}
The characteristic equation of (13) is
$$r^{2}-4r+4=0$$
The roots are $r_{1,2}=2$ 
\\Hence, the general solution is 
$$y(t)=\Bigg(c_{1}+c_{2}\frac{\Gamma(\beta+1)}{\alpha}t^{\alpha}\Bigg)e^{\frac{2\Gamma(\beta+1)}{\alpha}t^{\alpha}}$$
\end{Ex}

\begin{Ex}
\begin{equation}
^{2}D^{\alpha,\beta}_{M}y+4D^{\alpha,\beta}_{M}y+5y=0
\end{equation}
The characteristic equation of (14) is
$$r^{2}+4r+5=0$$
The roots are $r_{1}=-2+i$ and $r_{2}=-2-i$
\\Hence, the general solution is 
$$y(t)=e^{\frac{-2\Gamma(\beta+1)}{\alpha}t^{\alpha}}\Big[cos\Big(\frac{\Gamma(\beta+1)}{\alpha}t^{\alpha}\Big)+isin\Big(\frac{\Gamma(\beta+1)}{\alpha}t^{\alpha}\Big)\Big]$$
\end{Ex}

\section{Solution of Non-Homogeneous Case}
In this section, Method of variation of parameters is applied to derive the particular solution of the equation. 
\begin{equation}
^{n}D^{\alpha,\beta}_{M}y+p_{n-1}(t)^{n-1}D^{\alpha,\beta}_{M}y+...+p_{2}(t)^{2}D^{\alpha,\beta}_{M}y+p_{1}(t)D^{\alpha,\beta}_{M}y+p_{0}(t)y=f(t)
\end{equation}
where $y$ is $n$ times \textit{M}-differentiable function for $\alpha \in(0,1]$  and $\beta>0$.

\begin{mythm}
If $u(t)$ is a solution of homogeneous case of the equation (15) such that 
\begin{equation}
u(t)=\sum_{i=1}^{n}c_{i} y_{i}(t)
\end{equation}
then particular solution of the equation (15) is
$$v(t)=\sum_{i=1}^{n}c_{i}(t)y_{i}(t)$$
Where $c_{1}(t),c_{2}(t),...,c_{n}(t)$ provide following system of equations
$$\sum_{i=1}^{n}D^{\alpha,\beta}_{M}c_{i}(t)y_{i}(t)=0$$
$$\sum_{i=1}^{n}D^{\alpha,\beta}_{M}c_{i}(t)D^{\alpha,\beta}_{M}y_{i}(t)=0$$
$$\vdots$$
$$\sum_{i=1}^{n}D^{\alpha,\beta}_{M}c_{i}(t)^{n-2}D^{\alpha,\beta}_{M}y_{i}(t)=0$$
$$\sum_{i=1}^{n}D^{\alpha,\beta}_{M}c_{i}(t)^{n-1}D^{\alpha,\beta}_{M}y_{i}(t)=f(t)$$
\end{mythm}
\begin{proof}
The solution of the equation (15) is in the form
$$v(t)=\sum_{i=1}^{n}c_{i}(t)y_{i}(t)$$
The \textit{M}-derivative of $v(t)$ for $\alpha \in (0,1]$ and $\beta>0$ will be
$$D^{\alpha,\beta}_{M}v(t)=\sum_{i=1}^{n}c_{i}(t)D^{\alpha,\beta}_{M}y_{i}(t)+ \sum_{i=1}^{n}D^{\alpha,\beta}_{M} c_{i}(t)y_{i}(t)$$
Applying the first condition $\sum_{i=1}^{n}D^{\alpha,\beta}_{M}c_{i}(t)y_{i}(t)=0$, we obtain
$$D^{\alpha,\beta}_{M}v(t)=\sum_{i=1}^{n}c_{i}(t)D^{\alpha,\beta}_{M}y_{i}(t)$$
If we calculate the \textit{M}-derivative of $D^{\alpha,\beta}_{M}v(t)$ for $\alpha\in(0,1]$ and $\beta>0$, then we get
$$^{2}D^{\alpha,\beta}_{M}v(t)=\sum_{i=1}^{n}c_{i}(t)^{2}D^{\alpha,\beta}_{M}y_{i}(t)+ \sum_{i=1}^{n}D^{\alpha,\beta}_{M} c_{i}(t)D^{\alpha,\beta}_{M}y_{i}(t)$$
Apply second condition $\sum_{i=1}^{n}D^{\alpha,\beta}_{M}c_{i}(t)D^{\alpha,\beta}_{M}y_{i}(t)=0$, we obtain
$$^{2}D^{\alpha,\beta}_{M}v(t)=\sum_{i=1}^{n}c_{i}(t)^{2}D^{\alpha,\beta}_{M}y_{i}(t)$$
By continuing in this way, we get
$$^{n-1}D^{\alpha,\beta}_{M}v(t)=\sum_{i=1}^{n}c_{i}(t)^{n-1}D^{\alpha,\beta}_{M}y_{i}(t)+ \sum_{i=1}^{n}D^{\alpha,\beta}_{M} c_{i}(t)^{n-2}D^{\alpha,\beta}_{M}y_{i}(t)$$
We substitute  $v(t),^{2}D^{\alpha,\beta}_{M}v(t),..., ^{n}D^{\alpha,\beta}_{M}v(t)$ in the equation (15), we have
\begin{multline*}
\sum_{i=1}^{n}c_{i}(t)^{n}D^{\alpha,\beta}_{M}y_{i}(t)+\sum_{i=1}^{n}D^{\alpha,\beta}_{M} c_{i}(t)^{n-1}D^{\alpha,\beta}_{M}y_{i}(t)+p_{n-1}\sum_{i=1}^{n} c_{i}(t)^{n-1}D^{\alpha,\beta}_{M}y_{i}(t)+...+
\\p_{2}\sum_{i=1}^{n}c_{i}(t)^{2}D^{\alpha,\beta}_{M}y_{i}(t)+p_{1}\sum_{i=1}^{n}c_{i}(t)D^{\alpha,\beta}_{M}y_{i}(t)+p_{0}\sum_{i=1}^{n} c_{i}(t)y_{i}(t)=f(t)
\end{multline*}
\begin{multline*}
\sum_{i=1}^{n}D^{\alpha,\beta}_{M}c_{i}(t)^{n-1}D^{\alpha,\beta}_{M}y_{i}(t)+\sum_{i=1}^{n}
c_{i}(t)\big[^{n}D^{\alpha,\beta}_{M}y_{i}(t)+p_{n-1}^{n-1}D^{\alpha,\beta}_{M}y_{i}(t)+...+\\p_{1}D^{\alpha,\beta}_{M}y_{i}(t)+p_{0}y_{i}(t)\big]=f(t)
\end{multline*}
Since $y_{1}(t),y_{2}(t),...,y_{n}(t)$ are solutions of homogeneous case of equation (8), then
$$\sum_{i=1}^{n}
c_{i}(t)\big[^{n}D^{\alpha,\beta}_{M}y_{i}(t)+p_{n-1}^{n-1}D^{\alpha,\beta}_{M}y_{i}(t)+...+\\p_{1}D^{\alpha,\beta}_{M}y_{i}(t)+p_{0}y_{i}(t)\big]=0$$
We obtain $n^{th}$ condition as
$$\sum_{i=1}^{n}D^{\alpha,\beta}_{M}c_{i}(t)^{n-1}D^{\alpha,\beta}_{M}y_{i}(t)=f(t)$$
Hence we obtain the following system
$$\sum_{i=1}^{n}D^{\alpha,\beta}_{M}c_{i}(t)y_{i}(t)=0$$
$$\sum_{i=1}^{n}D^{\alpha,\beta}_{M}c_{i}(t)D^{\alpha,\beta}_{M}y_{i}(t)=0$$
\begin{equation}
\vdots
\end{equation}
$$\sum_{i=1}^{n}D^{\alpha,\beta}_{M}c_{i}(t)^{n-2}D^{\alpha,\beta}_{M}y_{i}(t)=0$$
$$\sum_{i=1}^{n}D^{\alpha,\beta}_{M}c_{i}(t)^{n-1}D^{\alpha,\beta}_{M}y_{i}(t)=f(t)$$
Solving the above system (17) provides $D^{\alpha,\beta}_{M}c_{i}(t)$,   $i=1,2,...,n$. Therefore we can write the particular solution of equation (15) as $v(t)=\sum_{i=1}^{n}c_{i}(t)y_{i}(t)$ 
\end{proof}

\begin{Ex}
 $$^{2}D^{\alpha,\beta}_{M}y+4D^{\alpha,\beta}_{M}y+3y=f(t)$$
 (a) Let $f(t)=e^{2t^{\alpha}}$. For $v(t)=c_{1}(t) e^{\frac{-3\Gamma(\beta+1)}{\alpha}t^{\alpha}}+c_{2}(t) e^{\frac{-\Gamma(\beta+1)}{\alpha}t^{\alpha}}$, the system of equations are built by the conditions as following
 $$D^{\alpha,\beta}_{M}c_{1}(t) e^{\frac{-3\Gamma(\beta+1)}{\alpha}t^{\alpha}}+D^{\alpha,\beta}_{M}c_{2}(t) e^{\frac{-\Gamma(\beta+1)}{\alpha}t^{\alpha}}=0$$
 $$-3D^{\alpha,\beta}_{M}c_{1}(t) e^{\frac{-3\Gamma(\beta+1)}{\alpha}t^{\alpha}}-D^{\alpha,\beta}_{M}c_{2}(t) e^{\frac{-\Gamma(\beta+1)}{\alpha}t^{\alpha}}=e^{2t^{\alpha}}$$
 Solving the above system of equations and using \textit{M}-integral we obtain  $c_{1}(t)=\frac{-\Gamma(\beta+1)}{4\alpha+6\Gamma(\beta+1)} e^{\frac{2\alpha+3\Gamma(\beta+1)}{\alpha}t^{\alpha}},c_{2}(t)=\frac{\Gamma(\beta+1)}{4\alpha+2\Gamma(\beta+1)} e^{\frac{2\alpha+\Gamma(\beta+1)}{\alpha}t^{\alpha}}$. Then particular solution $v(t)$ is
$$ v(t)=\frac{\Gamma(\beta+1)^{2}}{4\alpha^{2}+8\alpha\Gamma(\beta+1)+3\Gamma(\beta+1)^{2}} e^{2t^{\alpha}}$$ 	
(b)	Let $f(t)=2t^{2\alpha}+t^{\alpha}-3$. The system of equations for this case is
$$D^{\alpha,\beta}_{M}c_{1}(t) e^{\frac{-3\Gamma(\beta+1)}{\alpha}t^{\alpha}}+D^{\alpha,\beta}_{M}c_{2}(t) e^{\frac{-\Gamma(\beta+1)}{\alpha}t^{\alpha}}=0$$
 $$-3D^{\alpha,\beta}_{M}c_{1}(t) e^{\frac{-3\Gamma(\beta+1)}{\alpha}t^{\alpha}}-D^{\alpha,\beta}_{M}c_{2}(t) e^{\frac{-\Gamma(\beta+1)}{\alpha}t^{\alpha}}=2t^{2\alpha}+t^{\alpha}-3$$
Solve this system of equations, we have
\begin{multline*}
c_{1}(t)=\frac{-1}{3}t^{2\alpha}e^{\frac{3\Gamma(\beta+1)}{\alpha}t^{\alpha}}+\Big(\frac{4\alpha-3\Gamma(\beta+1)}{18\Gamma(\beta+1)}\Big) t^{\alpha} e^{\frac{3\Gamma(\beta+1)}{\alpha}t^{\alpha}}\\
+\Bigg(\frac{-4\alpha^{2}+3\alpha\Gamma(\beta+1)+27\Gamma(\beta+1)^{2}}{54\Gamma(\beta+1)^{2}}\Bigg)e^{\frac{3\Gamma(\beta+1)}{\alpha}t^{\alpha}}
\end{multline*}
\begin{multline*}
c_{2}(t)=t^{2\alpha}e^{\frac{\Gamma(\beta+1)}{\alpha}t^{\alpha}}+\Big(\frac{\Gamma(\beta+1)-4\alpha}{2\Gamma(\beta+1)}\Big) t^{\alpha} e^{\frac{\Gamma(\beta+1)}{\alpha}t^{\alpha}}\\
+\Bigg(\frac{4\alpha^{2}-\alpha\Gamma(\beta+1)-3\Gamma(\beta+1)^{2}}{2\Gamma(\beta+1)^{2}}\Bigg)e^{\frac{\Gamma(\beta+1)}{\alpha}t^{\alpha}}
\end{multline*}
Hence, particular solution $v(t)$ is obtained by
$$v(t)=\frac{2}{3}t^{2\alpha}+\Big(\frac{3\Gamma(\beta+1)-16\alpha}{9\Gamma(\beta+1)}\Big) t^{\alpha} e^{\frac{\Gamma(\beta+1)}{\alpha}t^{\alpha}}\\
+\Bigg(\frac{52\alpha^{2}-12\alpha\Gamma(\beta+1)-27\Gamma(\beta+1)^{2}}{27\Gamma(\beta+1)^{2}}\Bigg)e^{\frac{\Gamma(\beta+1)}{\alpha}t^{\alpha}}$$

(c)	Let $f(t)=sin2t^{\alpha}$. The system of equations for this case is
$$D^{\alpha,\beta}_{M}c_{1}(t) e^{\frac{-3\Gamma(\beta+1)}{\alpha}t^{\alpha}}+D^{\alpha,\beta}_{M}c_{2}(t) e^{\frac{-\Gamma(\beta+1)}{\alpha}t^{\alpha}}=0$$
 $$-3D^{\alpha,\beta}_{M}c_{1}(t) e^{\frac{-3\Gamma(\beta+1)}{\alpha}t^{\alpha}}-D^{\alpha,\beta}_{M}c_{2}(t) e^{\frac{-\Gamma(\beta+1)}{\alpha}t^{\alpha}}=sin2t^{\alpha}$$
Solve this system of equations, we have
$$c_{1}(t)=\frac{-3\Gamma(\beta+1)^{2}}{8\alpha^{2}+18\Gamma(\beta+1)^2}  e^{\frac{3\Gamma(\beta+1)}{\alpha}t^{\alpha}}sin2t^{\alpha}
+\frac{\alpha\Gamma(\beta+1)}{4\alpha^{2}+9\Gamma(\beta+1)^{2}}e^{\frac{3\Gamma(\beta+1)}{\alpha}t^{\alpha}}cos2t^{\alpha}$$
$$c_{2}(t)=\frac{\Gamma(\beta+1)^{2}}{8\alpha^{2}+2\Gamma(\beta+1)^2}  e^{\frac{\Gamma(\beta+1)}{\alpha}t^{\alpha}}sin2t^{\alpha}\\
-\frac{\alpha\Gamma(\beta+1)}{4\alpha^{2}+\Gamma(\beta+1)^{2}}e^{\frac{3\Gamma(\beta+1)}{\alpha}t^{\alpha}}cos2t^{\alpha}$$
Hence, particular solution $v(t)$ is obtained by 
\begin{multline*}
v(t)=\frac{-4\alpha^{2}\Gamma(\beta+1)^{2}+3\Gamma(\beta+1)^{4}}{16\alpha^{4}+40\alpha^{2}\Gamma(\beta+1)^{2}+9\Gamma(\beta+1)^4}sin2t^{\alpha}\\
-\frac{8\alpha\Gamma(\beta+1)^{3}}{16\alpha^{4}+40\alpha^{2}\Gamma(\beta+1)^{2}+9\Gamma(\beta+1)^4}cos2t^{\alpha}
\end{multline*}
(d)	Let $f(t)=e^{2t^{\alpha}}t^{\alpha}$. The system of equations for this case is
$$D^{\alpha,\beta}_{M}c_{1}(t) e^{\frac{-3\Gamma(\beta+1)}{\alpha}t^{\alpha}}+D^{\alpha,\beta}_{M}c_{2}(t) e^{\frac{-\Gamma(\beta+1)}{\alpha}t^{\alpha}}=0$$
 $$-3D^{\alpha,\beta}_{M}c_{1}(t) e^{\frac{-3\Gamma(\beta+1)}{\alpha}t^{\alpha}}-D^{\alpha,\beta}_{M}c_{2}(t) e^{\frac{-\Gamma(\beta+1)}{\alpha}t^{\alpha}}=e^{2t^{\alpha}}t^{\alpha}$$
Solve this system of equations, we have
$$c_{1}(t)=\frac{-\Gamma(\beta+1)}{4\alpha+6\Gamma(\beta+1)}t^{\alpha}e^{\frac{2\alpha+3\Gamma(\beta+1)}{\alpha}t^{\alpha}}  +\frac{\alpha\Gamma(\beta+1)}{2\big(2\alpha+3\Gamma(\beta+1)\big)^{2}}e^{\frac{2\alpha+3\Gamma(\beta+1)}{\alpha}t^{\alpha}}$$
$$c_{2}(t)=\frac{\Gamma(\beta+1)}{4\alpha+2\Gamma(\beta+1)}t^{\alpha}e^{\frac{2\alpha+\Gamma(\beta+1)}{\alpha}t^{\alpha}}  -\frac{\alpha\Gamma(\beta+1)}{2\big(2\alpha+\Gamma(\beta+1)\big)^{2}}e^{\frac{2\alpha+\Gamma(\beta+1)}{\alpha}t^{\alpha}}$$
Hence, particular solution $v(t)$ is obtained by
$$v(t)=\frac{\Gamma(\beta+1)^{2}}{4\alpha^{2}+8\alpha\Gamma(\beta+1)+3\Gamma(\beta+1)^2}t^{\alpha}e^{2t^{\alpha}}-\frac{4\alpha^{2}\Gamma(\beta+1)^{2}+4\alpha\Gamma(\beta+1)^{3}}{\big(4\alpha^{2}+8\alpha\Gamma(\beta+1)+3\Gamma(\beta+1)^{2}\big)^{2}}e^{2t^{\alpha}}$$

(e)	Let $f(t)=e^{-4t^{\alpha}}$. Take $\alpha\neq \frac{3}{4}$ and $\alpha\neq \frac{1}{4}$, the system of equations for this case is
$$D^{\alpha,\beta}_{M}c_{1}(t) e^{\frac{-3\Gamma(\beta+1)}{\alpha}t^{\alpha}}+D^{\alpha,\beta}_{M}c_{2}(t) e^{\frac{-\Gamma(\beta+1)}{\alpha}t^{\alpha}}=0$$
 $$-3D^{\alpha,\beta}_{M}c_{1}(t) e^{\frac{-3\Gamma(\beta+1)}{\alpha}t^{\alpha}}-D^{\alpha,\beta}_{M}c_{2}(t) e^{\frac{-\Gamma(\beta+1)}{\alpha}t^{\alpha}}=e^{-4t^{\alpha}}$$
Solve this system of equations, we have
$$c_{1}(t)=\frac{\Gamma(\beta+1)}{8\alpha-6\Gamma(\beta+1)}e^{\frac{3\Gamma(\beta+1)-4\alpha}{\alpha}t^{\alpha}}$$
$$c_{2}(t)=\frac{\Gamma(\beta+1)}{2\Gamma(\beta+1)-8\alpha}e^{\frac{\Gamma(\beta+1)-4\alpha}{\alpha}}t^{\alpha}$$
Hence, we obtain particular solution $v(t)$ as following:
$$v(t)=\frac{\Gamma(\beta+1)^{2}}{16\alpha^{2}-16\alpha\Gamma(\beta+1)+3\Gamma(\beta+1)^{2}}e^{-4t^{\alpha}}$$
Take $\alpha=\frac{3}{4}$ and $v(t)=c_{1}(t)e^{-4\Gamma(\beta+1)t^{\frac{3}{4}}}+c_{2}(t) e^{\frac{-4\Gamma(\beta+1)}{3}t^{\frac{3}{4}}}$
\\The system of equations is
$$D^{\frac{3}{4},\beta}_{M}c_{1}(t) e^{-4\Gamma(\beta+1)t^{\frac{3}{4}}}+D^{\frac{3}{4},\beta}_{M}c_{2}(t) e^{\frac{-4\Gamma(\beta+1)}{3}t^{\frac{3}{4}}}=0$$
 $$-3D^{\frac{3}{4},\beta}_{M}c_{1}(t) e^{-4\Gamma(\beta+1)t^{\frac{3}{4}}}-D^{\frac{3}{4},\beta}_{M}c_{2}(t) e^{\frac{-4\Gamma(\beta+1)}{3}t^{\frac{3}{4}}}=e^{-4t^{\frac{3}{4}}}$$
We solve the above equation, $c_{1}(t)=-\frac{\Gamma(\beta+1)}{-6+6\Gamma(\beta+1)} e^{\big(-4+4\Gamma(\beta+1)\big)t^{\frac{3}{4}}}$ and
$c_{2}(t)=\frac{\Gamma(\beta+1)}{-6+2\Gamma(\beta+1)}e^{\big(-12+4\Gamma(\beta+1)\big)t^{\frac{3}{4}}}$ is obtained.\\
The particular solution is $v(t)=\frac{\Gamma(\beta+1)^{2}}{9-12\Gamma(\beta+1)+3\Gamma(\beta+1)^{2}}e^{-4t^{\frac{3}{4}}}$

Take $\alpha=\frac{1}{4}$ and $v(t)=c_{1}(t)e^{-12\Gamma(\beta+1)t^{\frac{1}{4}}}+c_{2}(t) e^{-4\Gamma(\beta+1)t^{\frac{1}{4}}}$
\\The system of equations is
$$D^{\frac{1}{4},\beta}_{M}c_{1}(t) e^{-12\Gamma(\beta+1)t^{\frac{1}{4}}}+D^{\frac{1}{4},\beta}_{M}c_{2}(t) e^{-4\Gamma(\beta+1)t^{\frac{1}{4}}}=0$$
 $$-3D^{\frac{1}{4},\beta}_{M}c_{1}(t) e^{-12\Gamma(\beta+1)t^{\frac{1}{4}}}-D^{\frac{1}{4},\beta}_{M}c_{2}(t) e^{-4\Gamma(\beta+1)t^{\frac{1}{4}}}=e^{-4t^{\frac{1}{4}}}$$
We solve the above equation,$ c_{1}(t)=-\frac{\Gamma(\beta+1)}{-2+6\Gamma(\beta+1)} e^{\big(-4+12\Gamma(\beta+1)\big)t^{\frac{1}{4}}}$ and
$c_{2}(t)=\frac{\Gamma(\beta+1)}{-2+2\Gamma(\beta+1)}e^{\big(-4+4\Gamma(\beta+1)\big)t^{\frac{1}{4}}}$ is obtained.\\
The particular solution is $v(t)=\frac{4\Gamma(\beta+1)^{2}}{4-16\Gamma(\beta+1)+12\Gamma(\beta+1)^{2}}e^{-4t^{\frac{1}{4}}}$
\end{Ex}

\section{Conclusion}
In this paper, Existences and Uniqueness theorems for sequential linear \textit{M}-fractional differential equations are presented. We give solution of  \textit{M}-fractional differential equations with constants for homogeneous case using 
fractional exponential function and for non homogeneous case, we applied
method of variation of parameters.

\end{document}